\documentclass[leqno]{amsart}
\usepackage{amssymb,latexsym,amsmath,amscd,epsfig,amsthm, amsfonts}
\usepackage{enumerate}

\makeatletter
\@namedef{subjclassname@2010}{%
  \textup{2010} Mathematics Subject Classification}
\makeatother

\newtheorem{theorem}{Theorem}[section]
\newtheorem{lemma}[theorem]{Lemma}
\newtheorem{proposition}[theorem]{Proposition}
\newtheorem{corollary}[theorem]{Corollary}

\newtheorem*{facta}{Theorem A}

\theoremstyle{remark}

\newtheorem*{remark}{Remark}

\newcommand{\C}{{\mathbf C}}

\newcommand{\D}{{\mathbf D}}

\newcommand{\N}{{\mathbf N}}
\newcommand{\R}{{\mathbf R}}

\newcommand {\hv} {h^\infty_v}
\newcommand {\kv} {k^\infty_v}

\newcommand{\bl}{{\mathcal B}}

\newcommand{\beq}{\begin{equation}}
\newcommand{\eeq}{\end{equation}}
\newcommand{\beqs}{\begin{equation*}}
\newcommand{\eeqs}{\end{equation*}}
\newcommand{\ba}{\begin{eqnarray}}
\newcommand{\ea}{\end{eqnarray}}
\newcommand{\bas}{\begin{eqnarray*}}
\newcommand{\eas}{\end{eqnarray*}}

\newcommand{\p}{\phi}

\newcommand{\rj}{\rho_j}

\newcommand{\fraction}{\frac}
\newcommand{\ls}{\lesssim}
\newcommand{\gs}{\gtrsim}

\begin{document}

\author
{Kjersti Solberg Eikrem} 
\title
{Hadamard gap series in growth spaces} 
\address{Department of Mathematical Sciences, Norwegian University of Science and
Technology, NO-7491 Trondheim, Norway}
\email{kjerstei@math.ntnu.no}

\begin{abstract}
Let $\hv$ be the class of harmonic functions in the unit disk  which admit a two-sided radial majorant $v(r)$. We consider functions $v $ that fulfill a doubling condition. We characterize functions in $\hv$ that are represented by Hadamard gap series in terms of their coefficients, and as a corollary we obtain  
a characterization  of Hadamard gap series in Bloch-type spaces for  weights with a doubling property.  We show that if $u\in\hv$ is represented by a Hadamard gap series, then $u $ will grow slower than $v$ or oscillate along almost all radii. We use the law of the iterated logarithm for trigonometric series to find an upper bound on the growth of a weighted average of the function $u $, and we show that the estimate is sharp.

\end{abstract}
\maketitle 

\section{Introduction}
\subsection{Growth spaces of harmonic functions}
Let $v$ be a positive increasing continuous function on $[0,1)$, assume that $v(0)=1$ and $\lim_{r\rightarrow 1} v(r)=+\infty$. We study the growth spaces of harmonic functions in the unit disk:
\begin{equation}
\label{eq:0}
h^\infty_v=\{u:\D\rightarrow\R, \Delta u=0, |u(z)|\le Kv(|z|)\ {\text{for some}}\ K>0\}.
\end{equation}
In this article we mainly deal with weights which satisfy the doubling condition  \begin{equation}
\label{vdouble}
v(1-d)\le Dv(1-2d).
\end{equation} 
For such weights we characterize functions in $h^\infty_v$ represented by Hadamard gap series. Further, using this characterization, we study radial behavior of such functions in more detail.

Growth spaces appear naturally as duals to the classical spaces of harmonic and analytic functions; a similar class of harmonic functions satisfying only one-sided estimate with $v(r)=\log\frac1{1-r}$ was introduced by B. Korenblum in \cite{K}. General growth spaces can be found in the works of L. Rubel and A. Shields, and  A. Shields and D. Williams, see \cite{RS,SW}. Multidimensional analogs were recently considered in \cite{AD,EM}.
Various results on coefficients of functions in growth spaces were obtained in \cite{BST}. 

\subsection{Hadamard gap series}
Hadamard gap series are functions of the form
\[f(z)=\sum_{k=1}^\infty a_{n_k}z^{n_k},\quad a_{n_k}\in\C,\]
where $n_{k+1}>\lambda n_k$ and $\lambda>1$; for the harmonic case we take $u=\Re f$. For various spaces of analytic and harmonic functions Hadamard gap series in these spaces can be characterized in terms of the coefficients $a_n$. For example, $\alpha$-Bloch spaces are studied in \cite{Y}, in \cite{Ste} more general Bloch-type spaces are considered, for $Q_K$ spaces results are obtained in \cite{WZ}, and the case of growth spaces of harmonic functions with normal weights (i.e. weights that essentially grow between two powers $(1-r)^{-a}$ and $(1-r)^{-b}$) is recently treated in \cite{XY}. All these results give necessary and sufficient conditions on the coefficients of the series in the form $|a_n|\le f(n)$ for an appropriate function $f$. For general growth spaces, in particular for those with slow growing function $v$, such one-term characterization does not exist. We prove that the condition 
\[\sum_{n_k\le N}|a_{n_k}|\le Cv(1-1/N)\]
characterizes all Hadamard gap series in $h^\infty_v$ for any $v$ that satisfies (\ref{vdouble}). This generalizes previous results for $v=\log\frac1{1-r}$ \cite{LM} and also one for normal weights \cite{XY}. The proof is based on a beautiful theorem of J.-P. Kahane, M.Weiss and G. Weiss \cite{KWW}. As a corollary we obtain the characterization of Hadamard gap series in Bloch-type spaces with weights satisfying a doubling condition.


\subsection{Radial oscillation} Boundary behavior of harmonic functions in growth spaces was studied in \cite{BLMT,LM,EM}. It turns out that a harmonic function $u\in h_v^\infty$ cannot grow as fast as $v$ along most of the radii, precise statements can be found in \cite{BLMT,EM}. In order to describe typical radial behavior of functions in $h_v^\infty$ for the Korenblum space with $v=\log\frac1{1-r}$ the following weighted average was introduced in \cite{LM}
\[
I_u(R,\phi)=\int_{1/2}^R\frac{u(re^{i\phi})}{(1-r)\left(\log\frac1{1-r}\right)^2}dr,\quad R\in (0,1),\quad \phi\in(-\pi,\pi].
\]       
The law of the iterated logarithm applied to this average gives an upper estimate for $I_u(R,\phi)$ and shows that in general it grows slower than the similar average of $|u(re^{i\phi})|$. This result is referred to as radial oscillation of harmonic functions in $h^\infty_v$, see \cite{LM} for details.

One of the aims of this article is to find an appropriate weighted average for other growth spaces $h^\infty_v$. We note that the argument in \cite{LM} cannot be applied for fast growing weights, though it seems plausible that the result still holds. We obtain radial oscillation for the case when $u\in h^\infty_v$ is represented by a Hadamard gap series, the only restriction on $v$ is the doubling condition (\ref{vdouble}). 
The precise statement is the following: Let $u\in h_v^\infty$ be a Hadamard gap series, define
\[
I^{(v)}_u(R,\phi)=\int_{1/2}^R\frac{u(re^{i\phi})dv(r)}{v^2(r)}.\]
Then for almost every $\phi$
\[
\lim_{R\rightarrow 1}\frac{I^{(v)}_u(R,\phi)}{\sqrt{\log v(R)\log\log\log v(R)}}\le C.\]

The proof uses the law of the iterated logarithm for trigonometric series from \cite{W}, see also \cite{SZ}. A straightforward estimate gives $|I^{(v)}_u(R,\phi)|\le K\log v(R)$; our result shows that the growth is much slower for almost every $\phi$. We will also give an example of a function $u\in h^\infty_v$ such that 
\[
\limsup \frac{I_{|u|}^{(v)}(R,\phi)}{\log v(R)}>0\]
for almost every $\phi$. This difference in the asymptotic behavior of $I_u^{(v)}$ and $I_{|u|}^{(v)}$ indicates oscillation. 

\subsection{One-sided estimates}
We consider also classes of harmonic functions satisfying one-sided estimate only,
\begin{equation}
\label{1}
k^\infty_v=\{u:\D\rightarrow\R, \Delta u=0, u(z)\le Kv(|z|)\ {\text{for some}}\ K>0\}.
\end{equation}
In general there is a substantial difference between classes $h_v^\infty$ and $k_v^\infty$, however it cannot be observed on functions represented by Hadamard gap series.

\subsection{Organization of the paper}
In section 2 we first prove an estimate on the coefficients of functions in $k^\infty_v$ for general $v$. For $v$ that fulfills the doubling condition we show that if $u$ is a Hadamard gap series then $u\in h_v^\infty$ if and only if $u\in k_v^\infty$ and characterize such series in terms of their coefficients. As a corollary we get a characterization of Hadamard gap series in general Bloch spaces. In section 3 we prove the result on radial oscillation and show that it is precise. 

For convenience we define a new function $g:[1,\infty)\rightarrow[1,\infty)$ such that $g(x)=v(1-x^{-1})$. Then (\ref{vdouble}) is equivalent to
\begin{equation}
\label{gdouble}
g(2x)\le Dg(x).
\end{equation}
We will keep this notation throughout the paper, and constants $K$ in (\ref{eq:0}), $D$ and $\lambda$ will preserve their identities. Other constants need not be the same in each place. From now on we will write $I_u(R,\phi)$ instead of $I^{(v)}_u(R,\phi)$ for simplicity.

\section{Coefficient estimates and Hadamard gap series }

\subsection{Coefficient estimates}
Each real harmonic function $u$ in the unit disk can be written as
\[
u(re^{i\phi})=\Re\sum_{n =0}^{\infty} a_n z^n, 
 \qquad a_n \in \C,\]
where 
\[  
a_n=\frac{r^{-n}}{\pi}\int_{-\pi}^\pi u(re^{i\phi})e^{-in\phi}d\phi,\qquad n\ge 0.\] 
We will now estimate the size of $|a_n | $ for $u \in \kv $. 

\begin{proposition}
\label{coefficient}
Assume that $u\in\kv$ and 
\[u(re^{i\phi})=\Re\sum_{n =0}^{\infty} a_n z^n, 
\qquad a_n \in \C.\]
Then $|a_n|\le C g ( n)$ for some $C =C (K) $, where $K $ is as in (\ref{1}).
\end{proposition}
This proposition is a generalization of Theorem 1 (i)  
   in \cite[p.~209]{K}, where $v (r) =\log \frac1{1 -r} $. It also generalizes Theorem 1.12 (a) in \cite{BST}, which is valid for $u\in\hv$ if $v$ satisfies \eqref{vdouble}.

\begin{proof}
We have
\[|a_n|=\left|\frac{r^{-n}}{\pi}\int_{-\pi}^\pi u(re^{i\phi})e^{-in\phi}d\phi\right|
\le\frac{r^{-n}}{\pi}\int_{-\pi}^{\pi}|u(re^{i\phi})|d\phi.\]
Note that 
\[\int_{-\pi}^{\pi}|u(re^{i\phi})|d\phi=\int_{u(re^{i\phi})\ge 0}u(re^{i\phi})d\phi-\int_{u(re^{i\phi})<0}u(re^{i\phi})d\phi=\]
\[
2\int_{u(re^{i\phi})\ge 0}u(re^{i\phi})d\phi-\int_{-\pi}^{\pi} u(re^{i\phi})d\phi\le 4\pi Kv (r)-2\pi u(0).\]
This implies that for any $r\in(0,1)$ 
\[|a_n|\le C_1 r^{-n} v (r)= C_1 r^{-n}g \left(\frac1{1-r}\right),\]
where $C_1$ depends on $K$. Let 
$r=1-1/n$, then 
$|a_n|\le C g ( n).$
\end{proof}

The estimate $|a_n|\le C g ( n)$ is in general not enough to imply that a function is in $\kv$, but in some cases it is, we will come back to that in Corollary \ref {fast}. 

\subsection{Characterization of  Hadamard gap series in growth spaces}

 We will prove the following: 
 \begin{theorem}
 \label{lacunary}
 Let $\{n_k\}_{k=1}^\infty$ be a sequence of positive integers such that $n_{k+1}\ge \lambda n_k$ for each $k$, where $\lambda>1$. Assume $v $ satisfies  (\ref{vdouble}) and let
 \beqs
 \label{eq:lacseries}
 u(z)=\Re\sum_k a_{n_k}z^{n_k}, \qquad a_{n_k} \in \C,
 \eeqs
 where the series converges in the unit disk.
Then  
$u\in\kv $ if and only if
there exists $\gamma$ such that $\sum_{n_k\le N}|a_{n_k}|\le \gamma \, g(N)$
for any $N\in \N$.
 \end{theorem}
This 
theorem generalizes Proposition 5.1 in \cite{LM}, which is based on results from \cite{K} and \cite{KWW}.  
\medskip

\begin {proof}
We first assume that $u\in\kv $. 
Let $r_N=2^{-1/N} $ 
and let $M $ be a constant.
 We have
\[
u(r_Ne^{i\phi})=\Re \sum_{n_k\le NM} a_{n_k}r_N^{n_k} e^{in_k\phi}+\Re\sum_{n_k> NM} a_{n_k}r_N^{n_k} e^{in_k\phi}=s_N(\phi)+t_N(\phi).\]
 By Proposition \ref{coefficient} 
\[ |t_N(\phi)|\le C\sum_{n_k>NM}g ( n_k) 2^{-n_k/N}.\] 
Since $u $ is a Hadamard gap series there exists $\lambda>1 $ such that $n_{j+1}>\lambda n_{j}$ for all $j $. We may assume that $\lambda \le 2$, and by adding extra terms with $a_j =0 $ we may also assume that $n_{j+1}\le 4n_{j} $. 
Let $n_k \geq NM $.
Then by (\ref{gdouble}) 
\begin{eqnarray}
\label{g}
 \frac{g (n_{k+1} ) 2^{- n_{k+1}/N} }{g (n_{k} ) 2^{- n_{k}/N}}\le  \frac{g (4n_{k})  }{g (n_{k} ) }2^{-n_k (\lambda-1)/N} 
 \le D^2 2^{-M(\lambda -1)} <q <1
\end{eqnarray}
when $M $ is 
sufficiently large. 
Using this and (\ref{gdouble}), the reminder term can then be estimated as follows
\[
|t_N(\phi)|\le C\sum_{n_k>NM}g ( n_k) 2^{-n_k/N}\le C_q g (NM) \le B 
g ( N).\]  
By Theorem I in \cite{KWW} there exists $\alpha=\alpha(\lambda)>0$ and $\phi_0\in(-\pi,\pi)$ such that
\[
s_N(\phi_0)	\ge\alpha\sum_{n_k\le NM}|a_{n_k}r_N^{n_k}|.\]
We have
$r_N \le 1 -\frac{1}{3N}$,  
so
$$ g \left(\frac1{1 -r_N}\right)\le g (3N)\le D^2 g (N).$$
Thus 
\begin{eqnarray*}
\sum_{n_k\le N}|a_{n_k}|&\le& 2\sum_{n_k\le NM}|a_{n_k}r_N^{n_k}|\le 2 \alpha^{-1}\left(u(r_N e^{i\phi_0})+|t_N(\phi_0)|\right)\\
& \le& 2 \alpha^{-1}\left(g \left(\frac1{1 -r_N}\right) +B g (N)\right) \le \gamma g ( N).
\end{eqnarray*}

 Now assume there exists $\gamma$ such that $\sum_{n_k\le N}|a_{n_k}|\le \gamma \, g(N)$
for any $N\in \N$. 
Let $r_N<r\le r_{N+1}$. 
Then 
  $\frac{N}{\log 2} \le \frac1{1 -r} $, and using (\ref{gdouble})  and (\ref{g}) we get 
\begin{eqnarray}
\label{abs}
 u(re^{i\phi}) 
 &\le& \sum_{n_k\le NM}|a_k|+\sum_{n_k>NM}\gamma g( n_k) r^{n_k}\le C_1 g (MN) \\
 &\le& C_2 g \left(\frac{N}{\log 2}\right) \le C_2 g \left( \frac{1}{1-r}\right).\nonumber
 \end{eqnarray}
\end{proof} 

\begin {corollary}
 If $u $ is a Hadamard gap series and $v $ fulfills \eqref {vdouble}, then $u\in\kv \Leftrightarrow u\in\hv$.
 \end {corollary}

\begin {proof}
If $\sum_{n_k\le N}|a_{n_k}|\le \gamma \, g(N)$ is true for $u$, then of course it is also true for $-u$. 
\end{proof}

\begin {corollary}
 Let $f $ be a holomorphic Hadamard gap series in the unit disk and let $v $ satisfy \eqref {vdouble}. If $\Re f \in\kv $ , then $\Re f,\Im f\in\hv $.
\end {corollary}

  For another result of this type, see \cite {C}.  There it is proven that if $v (r)=\left (\fraction {1} {1 -r}\right)^\alpha$ for $\alpha>1 $ and $f $ is holomorphic, then $\Re f \in\kv $ implies $\Re f \in\hv $ and $\Im f\in\hv $.
For some $v $ the fact that $u\in\kv $ will imply that $u\in\hv $, see \cite {N} and \cite {B}.

\begin {corollary}
\label {fast}
If   $u $ is a Hadamard gap series and $g $ is such that for each $q>1 $ there is $A (q)>1 $ that satisfies 
\beq
\label{implication}
x>q y\,\Rightarrow \, g (x)>A g (y)
\eeq
 when $y>y_0 $, 
 then
 $|a_{n_k}|\le C g (n_k) $ implies $u\in\hv$.
\end {corollary}
\begin {proof}
We get
$$\sum_{n_k\le N}|a_{n_k}|\le C\sum_{n_k\le N}g (n_k)\le C_1  g( N)$$ 
whenever $\fraction{n_{k+1}} {n_k}\ge \lambda>1 $, where $C_1 $ depends on $A$.  
\end {proof}

For example when $g (x) =x^{\gamma},\gamma> 0 $, (\ref{implication}) is satisfied. This corollary is slightly more general than 
Theorem 2.3 in \cite {XY}, which concerns 
normal weights.

For general $v $, the inequality $|a_{n}|\le C g ( n)$ does not imply that $u\in\kv$ even when $u$ is represented by a Hadamard gap series. 
It is not difficult to construct $g $ (for example $g (x) =\log x $) such that for any $C $ 
$$\sum_{ k\le n} g (2^k)>C g (2^n) \qquad \mathrm{ for\; some}\; n =n (C) .  $$ 
 Then by  
 the previous theorem  $u(z) =\sum_{k =1}^\infty g (2^k)z^{2^k} \notin\kv.  $


\subsection{Application to Bloch-type functions}
Let $\mu $ be a decreasing continuous function on $[0,1) $ and assume $\lim_ {r\rightarrow 1^-} \mu(r)=0$.
A harmonic function is in the Bloch-type space $\bl_\mu$ if
$$| |u ||_{\bl_\mu} =\sup_{z\in\D} ( |u (0) |+ \mu(|z|)|\nabla u (z)|)  <\infty. $$
Hadamard gap series in  Bloch-type  spaces have been characterized in terms of their coefficients  for many concrete $\mu $. For $\mu(r)=(1-r^2)^\alpha$, 
a Hadamard gap series is  
in $\bl_\mu$ if  and only if
$\limsup_{j \rightarrow \infty} |a_j|n_j^{1-\alpha}  <\infty$, this is from \cite{Y}.  In \cite {Ste} 
the condition $$\limsup_{j \rightarrow \infty} n_j |a_j|\mu (1 -1/n_j) <\infty$$  is obtained under some restrictions on the weight,  in particular for the weight $\mu(r)=(1-r)\Pi_{j =1}^k \log^{[j]}\fraction {e^k}{1-r}$, where $\log^{[j]} $ is the logarithm applied $j $ times.
Using Theorem \ref {lacunary} we will obtain a more general result, where we only assume that $\mu $ does not decay too fast, more presicely, that it fullfills a condition similar to \eqref{vdouble}:
\beq
\label {wdouble}
 \mu\left(1-\fraction {d} {2}\right) \ge B \mu(1-d).
\eeq

\begin {corollary}
\label {Bloch function}
Let $\mu $ fulfill \eqref{wdouble}.
A Hadamard gap series is in the Bloch-type space $\bl_{\mu} $ 
if and only if there is a $C $  such that
$$\sum_{n_k\le N}n_k|a_{n_k}|\le C \, \fraction {1} {\mu (1 -1/N)} .  $$ 
\end {corollary}

\begin{proof}
A function $ u(z)=\Re\sum_k a_{n_k}z^{n_k}$ is in $ \bl_\mu$ if and only if $\Re\sum_k n_k a_{n_k}z^{n_k} $ is in  $\hv $ for $v (r) =1/\mu (r) $. Then the result follows from Theorem \ref {lacunary}.
\end {proof}

\begin {corollary}
Let $\mu $ fulfill \eqref{wdouble}.
If   $u $ is a Hadamard gap series and $\mu $ is such that for each $q>1 $ there is $A(q)>1 $ that satisfies 
\beq
\label{implication w}
x>q y\,\Rightarrow \, \mu (1-1/y)>A \mu (1-1/x)
\eeq
 when $y>y_0 $, 
 then 
  $\sup_k n_k|a_{n_k}|\mu (1 -1/n_k)\le C $
  implies $u\in\bl_\mu$.
\end {corollary}

\begin{proof}
This follows from   Corollary \ref {fast} and \ref{Bloch function}.
\end {proof}


\subsection{Examples}
We give examples of functions in  $\hv $ when $v $ 
 fulfills the doubling condition \eqref{vdouble}. These examples will be used later.
Let $A>1 $, $b_0 =1 $ and define $b_n $ by induction as 
\beq
\label {b}
b_{n+1}=\min \{l \in \N: g(2^l)>A g(2^{b_n})\}.
\eeq

\begin{lemma}
\label{first example}
Assume $v $ satisfies (\ref{vdouble}).
Let  
\[
u(z)= \,\Re \sum_{k=0}^\infty g (2^{b_k}) z^{2^{b_k}}, \qquad z \in \D ,
\]
where $b_n $ is given as above.
Then $u \in \hv $.
\end{lemma}

\begin{proof}
We have
\begin{eqnarray*}
   \sum_{n =0}^{N } g (2^{b_n} )   \leq  g (2^{b_{N}})\sum_{n =0}^{N}\frac{1}{A^{n}}\le \fraction{A}{A-1}\, g(2^{b_{N}}) 
\end{eqnarray*}
Then by (iii) of the previous theorem  $u \in \hv $. 
\end{proof}

 When $g (t) = t^a $ or $g (t) = (\log t)^a $,  examples of function in $\hv $ are respectively
 \[
u(z)= \,\Re\sum_{k=0}^\infty 2^{ka} z^{2^{k}}\quad \mathrm{and} \quad u(z)= \,\Re\sum_{k=0}^\infty 2^{ka} z^{2^{2^{k}}}.
\]


\section {Oscillation }
\subsection {A preliminary estimate}
We will need the following lemma to prove our main result.
By the relation $a\ls b$ we mean that there exists a constant $C $ that only depends on $D, K$ and $\lambda$ such that $a\le Cb $. If $a\ls b$ and $a\gs b$, we write $a\simeq b$.  The function $v $ is as in the introduction.
\begin {lemma}
\label {integral}
Let $v $ satisfy \eqref{vdouble}. There exist $C $ and $n_0 $, which depend only on $D $, such that
\beqs
\int_0^{s} \fraction { r^{n-1}}{v(r)}dr\le C\fraction {s^{n}} { n v(s)}
\eeqs
for $n\ge n_0$ and $ s\le 1-\fraction1 {n} $.
\end{lemma}
When we use $\ls $ in the proof of this lemma, the constants will only depend on $D $.
\begin {proof}
Let $l $ be such that 
\beq
\label {l}
1 +2\log (2 D^l) \le 2^l,
\eeq 
so $l $ only depends on $D $. Let $\rho_0 =0 $ and choose $\rj $ such that 
$v (\rj) =\gamma v (\rho_{j -1} )$  where $\gamma =D^l$.
Choose $k $ such that 
\beq
\label {kn}
\rho_{k}\le s
\le \rho_{k+1} .
\eeq
 Since $v (r)\ge \gamma^m $ when $r\ge \rho_m $, we have
\beq
\label {first part}
\int_0^{s 
} \fraction { r^{n-1}}{v(r)}dr\le\sum_{m=0}^{k-1}   \fraction 1 {\gamma^m}\int_{\rho_{m}}^{\rho_{m+1}}r^{n-1} dr+ \fraction 1 {\gamma^k}\int_{\rho_{k}}^{s}r^{n-1} dr\le\fraction1 {n}\left (\sum_{m=0}^{k-1}   \fraction {\rho_{m+1}^{n}} {\gamma^m}+\fraction {s^{n}} {\gamma^k}\right).
\eeq
We will now show that
\beq
\label {gamma}
\sum_{m=0}^{k-1}  \fraction {\rho_{m+1}^{n}} {\gamma^m}\ls   \fraction {\rho_{k}^{n}} {\gamma^{k - 1}}\ls \fraction {s^{n}} {\gamma^k}.
\eeq
This is true if 
$$2\fraction  {\rho_{m}^{n}} {\gamma^{m-1}} \le \fraction {\rho_{m+1}^{n}} {\gamma^m} $$
for any $ 1\le m\le k-1$.
That is equivalent to $(2\gamma)^{1/n}\rho_{m}\le \rho_{m+1} $. Assume $ s\le 1-\fraction1 {n}$, then by \eqref{kn} we have $1/n\le 1 -\rho_{k} $, and for $n \ge n_0$ for some $n_0 $ large enough,  
$$(2\gamma)^{1/n} =e^{\log (2\gamma)/n}\le 1 +2\log (2\gamma)/n\le  1 +2\log (2\gamma) (1 -\rho_{k}). $$
We need
$$( 1 +2\log (2\gamma) (1 -\rho_{k})) \rho_{m-1}\le \rho_{m} $$
for any $k $ and $m\le k $.
It is enough to check that
$$( 1 +2\log (2\gamma) (1 -\rho_{k})) \rho_{k-1}\le \rho_{k}. $$
(This can be seen by fixing $m $ and letting $k $ vary.)
We rewrite this as
\beq
\label {rhok}
\rho_{k}\ge \fraction {(1 +2\log (2\gamma)) \rho_{k-1}}{1 +2\log (2\gamma) \rho_{k-1}}=1 -\fraction {1 -\rho_{k-1}}{1 +2\log (2\gamma)\rho_{k-1}}.
\eeq
We have
\beq
\label {first part rhok} 
v (\rho_{k})=\gamma v (\rho_{k-1}) =\gamma v (1 - (1 -\rho_{k-1})) 
\eeq
and by \eqref{vdouble} and \eqref{l},
\beq
\label {second part rhok}
 v \left(1 -\fraction {1 -\rho_{k-1}}{1 +2\log (2 \gamma)\rho_{k-1}}\right)\le \gamma v (1 - (1 -\rho_{k-1})).
 \eeq
 Then \eqref {first part rhok} and \eqref {second part rhok} give \eqref {rhok} since $v $ is increasing,
 and therefore \eqref{gamma} is true.
Furthermore,
\beq
\label {gamma2}
\fraction {s^{n}} {\gamma^k}\ls\fraction {s^{n}}{v (s)} .
\eeq
Then by  \eqref {first part}, \eqref {gamma} and 
\eqref {gamma2}, 
we get
$$\int_0^{s} \fraction { r^{n-1}}{v(r)}dr \ls \fraction{s^{n}}{n v (s)}  $$
when $n\ge n_0$ and $ s\le 1-\fraction1 {n} $. 
\end {proof}

\subsection {Main result} 
The following theorem by Mary Weiss will be used to prove our next theorem:

\begin{facta} 
Let $$S(x)=\sum_{k=0}^\infty a_j \cos (n_j\p)+\sum_{k=0}^\infty b_j \sin (n_j\p)$$
where $n_{j+1}/n_j>\lambda>1 $ for all $j $. 
 We write
\begin{gather*}
B_N = \left(\fraction12 \sum_{j=0}^N (a_j^2 +b_j^2)\right)^{1/2}, \qquad M_N =\max_{ 1\le j\le N} (a_j^2 +b_j^2)^{1/2},\\
S_N (\p) =\sum_{j=0}^N a_j    \cos n_j\p +\sum_{j=0}^N b_j \sin (n_j\p).
\end{gather*}
If 
 $B_N\rightarrow \infty $ and $M_N =o (B_N/(\log\log B_N)^{1/2})$ as  $N\rightarrow \infty $, then 
$$\limsup_{N\rightarrow \infty}\fraction {S_N (\p)} {(2B_N^2 \log\log B_N)^{1/2}}=1\qquad \mathrm{for\,a.e.\,}\p.  $$
\end{facta}

For $u\in\kv $ we define the weighted average
$$I_u (R,\p) =\int_{1/2}^{R} \fraction {u (re^{i\p})dv(r)} { (v (r))^2} .$$
See for example \cite  {St} for the definition of a Riemann-Stieltjes integral.
It follows from \eqref{1} that
$$|I_u (R,\p)| \le K \log v (R).  $$
We will prove the following:
\begin {theorem}
\label {main}
If $v $ satisfies \eqref{vdouble} and $u\in\hv $ is represented by a Hadamard gap series, then
\beq
\label  {the result}
 \limsup_{R\rightarrow 1}\fraction {I_u (R,\p)} {(\log v (R) \log\log \log v (R))^{1/2}} \le C
\eeq
for almost all $\p $, where $C$ depends on $D, K$ and $\lambda$.
\end {theorem}

\begin {proof}
Let
$$u (re^{i\p})   
=\sum_{j=0}^\infty \alpha_j r^{n_j}\cos (n_j\p)+\sum_{j=0}^\infty \beta_j r^{n_j}\sin (n_j\p)$$
where $\alpha_j ,\beta_j\in \R $,
then
$$I_u (R,\p) =\sum_{j=0}^\infty \alpha_j\int_{1/2}^{R} \fraction { r^{n_j}dv(r)} { (v (r))^2} \cos n_j\p+\sum_{j=0}^\infty \beta_j \int_{1/2}^{R} \fraction { r^{n_j}dv(r)} { (v (r))^2} \sin n_j\p.  $$
For simplicity we will only prove the result for the sum with cosines since the proof with sines is the same, so from now on we let $u (re^{i\p})   
=\sum_{k=0}^\infty \alpha_j r^{n_j}\cos (n_j\p) $ and 
$$I_u (R,\p) =\sum_{j=0}^\infty \alpha_j \int_{1/2}^{R} \fraction { r^{n_j}dv(r)} { (v (r))^2} \cos n_j\p. $$
Since $u $ is a Hadamard gap series there exists $\lambda>1 $ such that $n_{j+1}>\lambda n_{j}$ for all $j $. As before we may assume that $\lambda \le 2$ 
and $ n_{j+1}\le 4n_{j} $.
Let 
$$c_j = \int_{1/2}^{1} \fraction { r^{n_j}dv(r)} { (v (r))^2} $$
and
$$S_N (\p) =\sum_{j=0}^N \alpha_j c_j  \cos n_j\p .$$
Let also $r_N =1 -\fraction1 {n_N} $ and suppose  $R\in [r_N,r_{N +1}) $. 
We will first show that
\beq
\label {difference}
\left |I_u (R,\p)  -S_N (\p)  \right|\ls 1.  
\eeq
By \eqref {1} and \eqref {vdouble} we have
\beq
\label {idifference}
\left |I_u (R,\p) -I_u (r_N,\p)\right |\le K \int_{ r_N}^{r_{N +1}} \fraction {dv  (r)} { v (r)} =K \log\fraction {v (r_{N +1})}{v (r_{N })}\ls 1.
\eeq
Moreover,
\bas
I_u (r_N,\p) &=&\sum_{j=0}^\infty \alpha_j \int_{1/2}^{r_N} \fraction {r^{n_j} dv(r)} { (v (r))^2} \cos n_j\p =\sum_{j=0}^\infty \alpha_j b_{j,N} \cos n_j\p 
\eas
and
\beqs
b_{j,N}\le \int_{0}^{r_N} \fraction { r^{n_j}dv(r)} { (v (r))^2}  =\left [\fraction {- r^{n_j}}{v(r)}\right]_0^{r_N}+n_j\int_{0}^{r_N}\fraction  {r^{n_j -1}}{v(r)}dr\le n_j\int_{0}^{r_N} \fraction { r^{n_j-1}}{v(r)}dr.
\eeqs
By Lemma \ref{integral},
$$b_{j,N} 
\ls \fraction{r_N^{n_j}}{ g(n_N)}  $$
when $n_j\ge \max \{n_N,n_0\} $.
Let $M $ be a constant such that
 $D^2 e^{-M(\lambda -1)} <q <1.$
By Proposition \ref {coefficient}, 
\beq
\label {nm}
\sum_{n_j>n_NM} |\alpha_j b_{j,N} |\ls \sum_{n_j>n_NM} g (n_j)\fraction{r_N^{n_j}}{ g(n_N)} \le \sum_{n_j>n_NM}  \fraction{g (n_j)}{ g(n_N)}e^{ -n_j/n_N} . 
\eeq
For $n_j>n_NM $ we have
\begin{eqnarray*}
 \frac{g (n_{j+1} )e^{- n_{j+1}/n_N} }{g (n_{j} ) e^{- n_{j}/n_N}}\le  \frac{g (4n_{j} )  }{g (n_{j} ) }e^{-n_j (\lambda-1)/n_N} 
 \le D^2 e^{-M(\lambda -1)} <q <1,
\end{eqnarray*}
 thus \eqref {nm} is bounded by a constant independent of $N $. 
Furthermore, by Theorem \ref {lacunary}, 
\beq
\label {middle}
\sum_{n_N<n_j\le n_NM}  |\alpha_j b_{j,N}  |\ls \sum_{n_N<n_j\le n_NM}   \fraction{| \alpha_j| r_N^{n_j}}{ g(n_N)} \le\fraction{1} { g(n_N)}\sum_{n_j\le n_NM}   |\alpha_j|\ls \fraction{g (n_NM)} {g (n_N)}\ls 1. 
\eeq
We now need to estimate
$\sum_{n_j\le n_N}  |\alpha_j (b_{j,N} -c_j )| $. For $j\le N$,
$$|b_{j,N} -c_j | =\int_{r_N}^{1} \fraction {r^{n_j} dv  (r)} { (v (r))^2}  \le\int_{r_N}^{1} \fraction {dv  (r)} { (v (r))^2}  = \fraction1 {v (r_N)}, $$
then  by Theorem~\ref {lacunary} and \eqref{gdouble}
\beq
\label {less}
\sum_{n_j\le n_N}  |\alpha_j ||b_{j,N} -c_j |\le \fraction1 {v (r_N)}\sum_{n_j\le n_N} |\alpha_j|\ls \fraction{g (n_N)} {g (n_N)}=  1.
\eeq
Thus \eqref{difference} follows from \eqref{idifference}, \eqref  {nm}  , \eqref  {middle} and \eqref {less}:

\bas
\left |I_u (R,\p)  -S_N (\p)  \right|\le \left |I_u (R,\p) -I_u (r_N,\p)\right |+ \left |I_u (r_N,\p)-S_N (\p)\right |\\
 \ls 1 +\sum_{n_j\le n_N}  |\alpha_j ||b_{j,N} -c_j | +\sum_{n_N<n_j\le n_NM}  |\alpha_j b_{j,N}  |+\sum_{n_j>n_NM} |\alpha_j b_{j,N} |\ls 1.
\eas

We claim that
\beq
\label {ac}
\sum_{j=0}^{N} (\alpha_j c_j)^2\ls\log g(n_N)
\eeq
for any $N $. 
Let $q_0 =1 $ and choose $q_j $ such that $g (q_j) =2 g (q_{j-1})$.  
When $n_N \in (q_p,q_{p+1}]$ we have $\log g(n_N)\ls p$, so  \eqref{ac} would follow if 
\beq
\label {ac2}
\sum_{q_k <n_j\le  q_{k+1}} (\alpha_j c_j)^2\ls 1
\eeq
for all $ k$.
We have
\beq
\label {cj estimate}
c_j\le \int_{0}^{1} \fraction { r^{n_j}dv(r)} { (v (r))^2} =\left [\fraction {- r^{n_j}}{v(r)}\right]_0^1+n_j\int_{0}^{1}\fraction  {r^{n_j -1}}{v(r)}dr =n_j\int_{0}^{1} \fraction { r^{n_j-1}}{v(r)}dr
\eeq
and
\beq
\label {second part}
\int_{1-1/n_j}^{1} \fraction { r^{n_j-1}}{v(r)} dr \le \fraction{ 1}{v(1-1/n_j)}\int_{1-1/n_j}^{1}  r^{n_j-1}dr \le \fraction{ 1}{n_j g(n_j)}.
\eeq
By Lemma \ref {integral},
\bas
\label {first part2}
\int_0^{1-1/n_j
} \fraction { r^{n_j-1}}{v(r)}dr\ls  \fraction{ 1}{n_j g(n_j)}
\eas
for $n_j\ge n_{j_0} $,
and by this, \eqref{cj estimate} and \eqref{second part} we get
\beq
\label {c_j}
c_j\ls  \fraction{ 1}{ g(n_j)}.
\eeq
Thus for $q_k <n_j\le  q_{k+1} $, we have $c_j\ls   1/ g(q_k)$. Then  by Theorem \ref {lacunary},
\beqs
\sum_{q_k <n_j\le  q_{k+1}} (\alpha_j c_j)^2\le   \fraction{ 1}{( g(q_k))^2}\left(\sum_{n_j\le  q_{k+1}} |\alpha_j| \right)^2\ls  \fraction{ (g ( q_{k+1}))^2}{( g(q_k))^2}\ls 1
\eeqs
 and \eqref {ac2} is proved.

Now let
$$B_N = \left(\fraction12 \sum_{j=0}^N (\alpha_j c_j)^2\right)^{1/2} $$
 and
$$M_N =\max_{ 1\le j\le N} |\alpha_j c_j|. $$
We want to use Theorem 1 in \cite {W} to show that
$$ \limsup_{N\rightarrow \infty}\fraction {S_N (\p)} {(\log g(n_N) \log\log \log g(n_N))^{1/2}} \ls 1$$
for almost all $\p $. Then \eqref{the result} will follow from  \eqref{difference}.

If $B_N $  is bounded, then $S_N $ is bounded a.e. (see for example Theorem 6.3 in \cite [p. 203]{Z}), so there is nothing to prove.
By Proposition \ref {coefficient} and \eqref {c_j} we have $M_N\ls 1 $, so if $B_N\rightarrow\infty $, 
the conditions of Theorem A are fulfilled. 
Since   $B_N \ls (\log g(n_N))^{1/2}$ by \eqref {ac}, we get
$$ \limsup_{N\rightarrow \infty}\fraction {S_N (\p)} {(\log g(n_N) \log\log \log g(n_N))^{1/2}} \ls \limsup_{N\rightarrow \infty}\fraction {S_N (\p)} {(2B_N^2 \log\log B_N)^{1/2}}=1$$
for almost all $\p $, and we are done.

\end {proof}

\begin {remark} The same result is in general not true if $u $ is replaced by $|u| $, as the next  lemma shows. This indicates that the function is oscillating.
\end {remark}

\begin {lemma}
Let $v$ satisfy \eqref{vdouble}. There is a Hadamard gap series $u\in\hv$ such that
\beq \label {I}
 \limsup_{R\rightarrow 1^-} \fraction {I_{|u |} (R,\phi)} {\log v(R)}
 >0 
 \eeq
for all $\phi $ in a set of positive measure.
\end{lemma}

\begin {proof} Let $u(z)= \,\Re \sum_{k=0}^\infty g (2^{b_k}) z^{2^{b_k}}$, where $b_k$  are defined by \eqref{b}.
By \eqref{1} it follows that $I_{|u |} (R,\phi)\le K\log v(R) $. Assume we have also shown that 
\beq\label {I big} | |I_{|u |} (R,\cdot) | |_1\ge c\log v(R) 
\eeq 
  for 
  all $R\ge R_0>1/2$.  
  Let 
  $$E_R =\{\phi\in (-\pi,\pi]:I_{|u |} (R,\phi)>\fraction {c} {2}\log v(R)\}, $$ then
\bas
| |I_{|u |} (R,\cdot) | |_1& =&\int_{E_R} I_{|u |} (R,\phi)d\phi +\int_{E_R^c} I_{|u |} (R,\phi)d\phi\\
&\le& |E_R |K \log v(R) + (1-|E_R |)\fraction {c} {2}\log v(R),
\eas
hence $|E_R|\ge c/( 2K -c)$.
Then there exists a set of positive measure such that \eqref{I} is fulfilled.

Now it remains to show \eqref {I big}.
Fix $r $ and choose $N $ such that $1-1/2^{b_N}\le r \le 1-1/2^{b_{N+1}} $. By using that  $g(2^{b_{k+1}})\le AD g(2^{b_k})$, we get
\bas
\int_0^{2\pi} |u (re^{i \phi}) |^2d\phi =\sum_{k=0}^\infty g (2^{b_k})^2r^{2^{b_k+1}} \ge g (2^{b_N})^2 \left(1-\frac{1}{2^{b_N}}\right)^{2^{b_N+1}}\ge\fraction{1} {16A^2D^2} (v (r))^2.
\eas
 Furthermore,
$$\int_0^{2\pi} |u (re^{i \phi}) |^2d\phi\le \max_\phi |u (re^{i \phi}) |\int_0^{2\pi} |u (re^{i \phi}) |d\phi\le Kv(r)\int_0^{2\pi} |u (re^{i \phi}) |d\phi,$$
hence
$$\int_0^{2\pi} |u (re^{i \phi}) |d\phi\ge \fraction {1} {16A^2D^2K} v (r)  
.  $$
Then
$$\int_0^{2\pi} I_{|u |} (R,\phi)d\phi=\int_{1/2}^{R} \int_0^{2\pi} |u (re^{i \phi}) |d\phi\fraction {dv(r)} { (v (r))^2} \ge c\log v (R)  $$
for all $R\ge R_0>1/2$, and we are done.
\end {proof}

\subsection{ Sharpness of Theorem \ref {main}}
 The estimate is precise: 
 There exists a function in $\hv$ such that for some $a >0$
\beq
\label  {sharpness}
 \limsup_{R\rightarrow 1}\fraction {I_u (R,\p)} {(\log v (R) \log\log \log v (R))^{1/2}} \ge a
\eeq
for almost all $\p\in (-\pi,\pi] $.

To see this, let $u $ be the function from Lemma \ref{first example} with $A = 2$. Then
$$I_u (R,\p) =\sum_{j=0}^\infty g(2^{b_j}) \int_{1/2}^{R} \fraction { r^{2^{b_j}}dv(r)} { (v (r))^2} \cos 2^{b_j}\p. $$
Let $r_j=1-\fraction1 {2^{b_j}} $. It can be shown that as in the proof of theorem \ref{main},  
$$\left |I_u (R,\p)  -S_N (\p)  \right|\ls 1 $$ when $R\in (r_N,r_{N +1}) $. 
We also have
\beqs
c_j \ge r_j^{2^{b_j}} \int_{ r_j}^{1} \fraction { dv(r)} { (v (r))^2} \gs \fraction1{v(r_j)},
\eeqs
so $c_j\simeq \fraction1{v(r_j)} $.
By this and \eqref{b}, 
$$B_N^2 = \fraction12 \sum_{j=0}^N (g(2^{b_j}) c_j)^2 \simeq N\simeq \log g(2^{b_N}) =\log v (r_N). $$
Then
\beqs
 \limsup_{R\rightarrow 1}\fraction {I_u (R,\p)} {(\log v (R) \log\log \log v (R))^{1/2}} \gs  \limsup_{N\rightarrow \infty}\fraction {S_N (\p)} {(B_N^2 \log\log B_N)^{1/2}} =1
\eeqs
for almost all $\phi$ by Theorem A.

\section*{Acknowledgements}
I want to thank my Ph.D. advisor Eugenia Malinnikova for suggesting this problem and
for valuable help.
\bibliographystyle{plain}
\bibliography{lacunarybib3}

\end{document}